\documentclass[11pt]{article}

\usepackage{indentfirst,latexsym,mathrsfs,amsmath,amssymb,amsthm,fancyhdr,txfonts,setspace}
\usepackage{graphicx, color}

\usepackage[T1]{fontenc}
\usepackage[utf8]{inputenc}

\newcommand{\sub}{\subset}

\newcommand{\Lom}[1]{L^{#1}(\Omega)}
\newcommand{\norm}[2][]{\|#2\|_{#1}}
\newcommand{\normm}[2]{\|#2\|_{#1}}
\newcommand{\na}{\nabla}
\newcommand{\delny}{\partial_{\nu}}
\newcommand{\amrand}{\big\rvert_{\partial\Omega}}
\newcommand{\Ombar}{\overline{\Omega}}
\newcommand{\kl}[1]{\left(#1\right)}
\newcommand{\Lap}{\Delta}
\newcommand{\dOm}{\partial\Omega}
\newcommand{\Tmax}{T_{\rm max}}
\newcommand{\ddt}{\frac{\mathrm{d}}{\mathrm{d}t}}
\newcommand{\intnT}{\int_0^T}
\newcommand{\eps}{\varepsilon}
\newcommand{\set}[1]{\left\{#1\right\}}
\newcommand{\ue}{u_{\eps }}
\newcommand{\ve}{v_{\eps }}
\newcommand{\uet}{u_{\eps t}}
\newcommand{\vet}{v_{\eps t}}

\newcommand{\ye}{y_{\eps }}
\newcommand{\weakto}{\rightharpoonup}
\newcommand{\weakstarto}{\stackrel{\ast}{\rightharpoonup}}
\newcommand{\calF}{\mathcal{F}}
\newcommand{\intninf}{\int_0^\infty}
\newcommand{\intnt}{\int_0^t}

\allowdisplaybreaks

\let\oldsection\section
\renewcommand\section{\setcounter{equation}{0}\oldsection}

\def\XXint#1#2#3{{\setbox0=\hbox{$#1{#2#3}{\int}$}
  \vcenter{\hbox{$#2#3$}}\kern-.5\wd0}}

\newcommand{\nn}{\nonumber}
\newcommand{\iO}{\int_\Omega}
\newcommand{\io}{\int_\Omega}
\newcommand{\Om}{\Omega}

\newcommand{\pa}{\partial}

\newcommand{\f}{\frac}

\newtheorem{thm}{Theorem}[section]
\newtheorem{lem}{Lemma}[section]

\newtheorem{dnt}{Definition}[section]
\newtheorem{remark}{Remark}[section]

\allowdisplaybreaks
\author{
Johannes Lankeit$^1$ \footnote{johannes.lankeit@math.uni-paderborn.de}\quad
Yulan
 Wang$^{2}$,\footnote{wangyulan-math@163.com}\\
\emph{\small 1. Institut f\"{u}r Mathematik, Universit\"{a}t
Paderborn, Warburger Straße 100, 33098 Paderborn, Germany}\\
\emph{ \small 2. School of Science, Xihua University, Chengdu
610039, China}} {\small }{\small }

\title{Global existence, boundedness and stabilization in a high-dimensional chemotaxis system with consumption}

\date{}

\begin{document}

\maketitle

\begin{abstract}
\noindent This paper deals with the homogeneous Neumann boundary-value problem for the
chemotaxis-consumption system
\begin{eqnarray*}
\left\{
\begin{array}{llc}
u_t=\Delta u-\chi\nabla\cdot\big(u\nabla v\big)+\kappa u-\mu u^2,
&x\in \Omega, \,t>0,\\
\displaystyle v_t=\Delta v-uv , &x\in \Omega,\, t>0,
 \end{array}
\right.
\end{eqnarray*}
in $N$-dimensional bounded smooth domains for suitably regular positive initial data. \\
We shall establish the existence of a global bounded
classical solution for suitably large $\mu$ and prove that for any $\mu>0$ there
exists a weak solution.\\
Moreover, in the case of $\kappa>0$ convergence to the constant equilibrium $(\f{\kappa}{\mu },0)$ is shown. 
\end{abstract}

\section{Introduction}
Chemotaxis is the adaption of the direction of movement to an external chemical signal. This signal can be a substance produced by the biological agents (cells, bacteria) themselves, as is the case in the celebrated Keller-Segel model (\cite{KS}, \cite{horstmann_I}) or -- in the case of even simpler organisms -- by a nutrient that is consumed. A prototypical model taking into account random and chemotactically directed movement of bacteria alongside death effects at points with high population densities and population growth together with diffusion and consumption of the nutrient is given by
\begin{align}\label{sys.intro}
 u_t&=\Delta u -\chi \nabla \cdot (u\nabla v) +\kappa u-\mu u^2\\
 v_t&=\Delta v - uv,\nn
\end{align}
considered in a smooth, bounded domain $\Om\subset \mathbb{R} ^N$ together with homogeneous Neumann boundary conditions and suitable initial data. Herein, $\chi > 0$, $\kappa\in\mathbb{R} $, $\mu >0$ denote chemotactic sensitivity, growth rate (or death rate, if negative) and strength of the overcrowding effect, respectively.
The system \eqref{sys.intro}, in a basic form often with $\kappa=\mu =0$, appears as part of chemotaxis-fluid models intensively studied over the past few years (see e.g. the survey \cite[sec. 4.1]{BBTW} or \cite{cao_lankeit} for a recent contribution with an extensive bibliography).

Compared with the classical Keller-Segel model
\begin{align}\label{KS}
 u_t&=\Delta u-\chi \na \cdot(u\na v)+\kappa u-\mu u^2\\
 v_t&=\Delta v-v+u, \nn
\end{align}
which we have given in the form with logarithmic source terms paralleling that in \eqref{sys.intro}, at first glance, \eqref{sys.intro} seems much more amenable to the global existence (und boundedness) of solutions -- after all, the second equation by comparison arguments immediately provides an $L^\infty$-bound for $v$.

However, such a bound is not sufficient for dealing with the chemotaxis term, and accordingly global existence and boundedness of solutions to \eqref{sys.intro} with $\kappa=\mu =0$ is only known under the smallness condition
\begin{equation}\label{smallnessconditionforsourcefreemodel}
  \chi \norm[\Lom\infty]{v(\cdot,0)}\leq \frac1{6(N+1)}
\end{equation}
on the initial data (\cite{tao_consumption_bdness}) or in a
two-dimensional setting (\cite{win_ctfluid}, \cite{win_arma} and
also \cite{xieli}). Their rate of convergence has been treated in
\cite{zhang_li}. In three-dimensional domains, weak solutions have
been constructed that eventually become smooth
\cite{taowin_ev_consumption}.

For \eqref{KS}, the presence of logarithmic terms has been shown to exclude otherwise possible finite-time blow-up phenomena (cf. \cite{win_blowuphigherdim}, \cite{mizoguchi_winkler_13}) -- at least as long as $\mu $ is sufficiently large if compared to the strenght of the chemotactic effects (\cite{winkler_10_boundedness})
 or if the dimension is $2$ (\cite{osaki_yagi_02}). If the quotient $\frac{\mu }{\chi }$ is sufficiently large, solutions to \eqref{KS} uniformly converge to the constant equilibrium (\cite{win_stability}); convergence rates have been considered in \cite{he_zheng}. Explicit largeness conditions on $\frac{\mu }{\chi }$ that ensure convergence, also for slightly more general source terms, can be found in \cite{lin_mu}, see also \cite{win_ksns_logsource}. For small $\mu >0$, at least global weak solutions are known to exist (\cite{lankeit_ev_smooth}), and in $3$-dimensional domains and for small $\kappa$, their large-time behaviour has been investigated (\cite{lankeit_ev_smooth}).

Also the chemotaxis-consumption model \eqref{sys.intro} has
already been considered with nontrivial source terms in
\cite{wang_khan_khan}. There it was proved that classical
solutions exist globally and are bounded as long as
\eqref{smallnessconditionforsourcefreemodel} holds -- which is the
same condition as for $\kappa=\mu =0$, thus shedding no light on
any possible interplay between chemotaxis and the population
kinetics.

In a three-dimensional setting and in the presence of a Navier-Stokes fluid, in \cite{lankeit_fluid} it was recently possible to construct global weak solutions for any positive $\mu $, which moreover eventually become classical and uniformly converge to the constant equilibrium in the large-time limit.

It is the aim of the present article to prove the existence of global classical solutions if only $\mu $ is suitably large and to show their large-time behaviour. For the case of small $\mu >0$, we will prove the existence of global weak solutions (in the sense of Definition \ref{def:weaksol}). \\

\noindent\textbf{What largeness condition on $\mu $ might be sufficient for boundedness?}
For the Keller-Segel type model \eqref{KS} the typical condition
reads: 'If $\mu $ is large compared to $\chi $, then the
solution is global and bounded, independent of initial data.' In
order to see why this condition would be far less natural for
\eqref{sys.intro}, let us suppose we are given suitably regular
initial data $u_0$, $v_0$ and a corresponding solution $(u,v)$ of
\[
 \begin{cases} u_t=\Delta u - \chi \nabla\cdot(u\nabla v) + \kappa u-\mu u^2\\
  v_t=\Delta v - uv\\
  \delny u\amrand=\delny v\amrand=0\\
  u(\cdot,0)=u_0, v(\cdot,0)=v_0,
 \end{cases}
\]
and let us define
\[
 w:=\chi  v.
\]
Then $(u,w)$ solves
\[
 \begin{cases} u_t=\Delta u - \nabla\cdot(u\nabla w) + \kappa u-\mu u^2\\
  w_t=\Delta w - uw\\
  \delny u\amrand=\delny v\amrand=0\\
  u(\cdot,0)=u_0, w(\cdot,0)=\chi v_0,
 \end{cases}
\]
which is the same system, only with different chemotaxis
coefficent and rescaled initial data for the second component.
Consequently, \textit{in
\eqref{sys.intro}, large initial data equal high chemotactic
strength}. Hence, there cannot be any condition for global
existence which includes $\mu $ and $\chi $, but not $\norm[\Lom
\infty]{v_0}$. In light of this discussion, the requirement in
Theorem \ref{thm1} that $\mu $ be large with respect to $\chi
\norm[\Lom \infty]{v_0}$ seems natural. On the other hand, this
observation does not preclude conditions that involve neither
$\chi $ nor $\norm[\Lom \infty]{v_0}$, and indeed $\mu >0$ is
sufficient for the global existence of weak solutions.

The first main result of the present article is global existence of classical solutions, provided that $\mu $ is sufficiently large as compared to $\norm[\Lom\infty]{\chi v_0}$:

\begin{thm}\label{thm1}
 Let $N\in \mathbb{N} $ and let $\Om\sub \mathbb{R} ^N$ be a smooth, bounded domain. There are constants $k_1=k_1(N)$ and $k_2=k_2(N)$ such that the following holds: Whenever $\kappa\in\mathbb{R} $, $\chi >0$, and $\mu >0$ and initial data
\begin{eqnarray}\label{id}
 \left\{
\begin{array}{llc}
u_0\in C^0(\overline\Omega), \quad { u_0> 0} \,\,\text{in}\,\,\Ombar,\\
\displaystyle
v_0\in C^1(\Ombar), \quad { v_0> 0 } \,\,\text{in}\,\,\Ombar
 \end{array}
\right.
\end{eqnarray}
 are such that
\[
 \mu >k_1(N)\norm[\Lom\infty]{\chi v_0}^{\frac1N} + k_2(N)\norm[\Lom\infty]{\chi v_0}^{2N},
\]
 then the system
 \begin{equation}\label{a}
\left\{
\begin{array}{llc}
u_t=\Delta u-\nabla\cdot\big(u\nabla v\big)+\kappa u-\mu u^2,
&x\in \Omega, \,t>0,\\
\displaystyle v_t=\Delta v-uv , &x\in \Omega,\, t>0,\\
 \displaystyle
\delny u=\delny v=0, &x\in\pa\Om,\, t>0,\\
\displaystyle u(x,0)=u_0(x),\,\, v(x,0)=v_0(x),
 &x\in\Om,
\end{array}
\right. \end{equation}
has a unique global classical solution $(u,v)$ which is uniformly bounded in the sense that there is some constant $C>0$ such that
\begin{eqnarray}\label{B}
\|u(\cdot,t)\|_{L^{\infty}(\Omega)}+\|v(\cdot,t)\|_{W^{1,\infty}(\Omega)}
 \le C \qquad
\mathrm{for}\,\,\mathrm{all} \quad t\in(0,\infty). \end{eqnarray}
\end{thm}

\begin{remark}\label{R1}
Here we have to leave open the question, whether, for small values of $\mu >0$ and large $\chi v_0$, blow-up of solutions is possible at all. Consequently, the range of $\mu$ in this
 result is not necessarily an optimal one. Nevertheless, the present condition can easily be made explicit (see Lemma  \ref{lem:ge.for.positive.eps.or.large.mu}
 and \eqref{eq:defk1k2}
 for the values of $k_1$ and $k_2$). It seems worth pointing out that, in contrast to the condition \eqref{smallnessconditionforsourcefreemodel}, Theorem \ref{thm1} admits large values of $\chi v_0$, if only $\mu $ is appropriately large.
\end{remark}

The second outcome of our analysis is concerned with the large time behaviour of global solutions and reads as follows:

\begin{thm}\label{thm3}
Let $N\in \mathbb{N} $ and let $\Om\sub\mathbb{R} ^N$ be a bounded smooth domain.
Suppose that $\chi >0$, $\kappa>0$ and $\mu>0$. Let $(u,v)\in C^{2,1}(\Ombar\times(0,\infty ))\cap C^0(\Ombar\times[0,\infty))$ be any global bounded solution to \eqref{a} (in the sense that \eqref{B} is fulfilled) which obeys \eqref{id}. Then
\begin{eqnarray}\label{stability1}
\Big\|u(\cdot,t)-\f{\kappa}{\mu }\Big\|_{L^{\infty}(\Omega)}\rightarrow
0
\end{eqnarray}
and
\begin{eqnarray}\label{stability2}
\|v(\cdot,t)\|_{L^{\infty}(\Omega)}\rightarrow 0
\end{eqnarray}
as $t\rightarrow \infty$.
\end{thm}

\begin{remark}
 This theorem in particular applies to the solutions considered in Theorem \ref{thm1}.
\end{remark}

\begin{remark}
 Boundedness is not necessary in the sense of \eqref{B}; in light of Lemma \ref{lem:ulp.to.boundedness}, the existence of $C>0$ and $p>N$ such that
 \[
  \norm[\Lom p]{u(\cdot,t)}\leq C\qquad \text{for all } t>0
 \]
 would be sufficient.
\end{remark}

\noindent\textbf{Unconditional global weak solvability.}
 As in the context of the classical Keller-Segel model \eqref{KS} (\cite{lankeit_ev_smooth}), global weak solutions to \eqref{a} can be shown to exist regardless of the size of initial data and for any positive $\mu$:

\begin{thm}\label{thm:weaksol}
 Let $N\in\mathbb{N} $ and let $\Om\sub \mathbb{R} ^N$ be a bounded smooth domain. Let $\chi >0$, $\kappa\in \mathbb{R} $, $\mu >0$ and assume that $u_0$, $v_0$ satisfy \eqref{id}. Then
 system \eqref{a} has a global weak solution (in the sense of Definition \ref{def:weaksol} below).
\end{thm}

These solutions, too, stabilize toward $(\f{\kappa}{\mu },0)$ as $t\to\infty$, even though in a weaker sense than guaranteed by Theorem \ref{thm3} for classical solutions:

\begin{thm}\label{thm:weaksol-limit}
 Let $N\in\mathbb{N} $ and let $\Om\sub \mathbb{R} ^N$ be a bounded smooth domain. Let $\chi >0$, $\kappa>0$, $\mu >0$ and assume that $u_0$, $v_0$ satisfy \eqref{id}. Then for any $p\in[1,\infty)$ the weak solution $(u,v)$ to \eqref{a} that has been constructed during the proof of Theorem \ref{thm:weaksol} satisfies
\[
 \norm[\Lom p]{v(\cdot,t)}\to 0 \quad \text{ and } \int_t^{t+1} \norm[\Lom 2]{u(\cdot,s)-\f{\kappa}{\mu }} ds \to 0
\]
as $t\to \infty$.
\end{thm}

\begin{remark}
 Under the restriction $N=3$, the existence of global weak solutions that eventually become smooth and uniformly converge to $(\f{\kappa}{\mu },0)$ has been proven in \cite{lankeit_fluid}, where a coupled chemotaxis-fluid model is treated.
\end{remark}

{\textbf{Plan of the paper.}} In Section \ref{sec:prelim} we will prepare some general calculus inequalities.
In the following for some $a>0$ we will then consider
\begin{equation}\label{epssys}
  \begin{cases} \uet=\Delta \ue - \chi \nabla\cdot(\ue\nabla \ve) + \kappa\ue - \mu \ue^2 - \eps  \ue^2\ln  a\ue \\
  \vet=\Delta \ve - \ue\ve\\
  \delny \ue\amrand=\delny \ve\amrand=0\\
  \ue(\cdot,0)=u_0, \ve(\cdot,0)=v_0.
 \end{cases}
\end{equation}
For $\eps =0$, this system reduces to \eqref{a}; for $\eps \in(0,1)$ we will be able to derive global existence of solutions without any concern for the size of initial data and hence obtain a suitable stepping stone for the construction of weak solutions. Beginning the study of solutions to this system in Section \ref{sec:locex-andbasic} with a local existence result and elementary properties of the solutions, we will in Section \ref{sec:bdclasssol} consider a functional of the type $\io u^p+\io |\na v|^{2p}$ and finally, aided by estimates for the heat semigroup, obtain globally bounded solutions, thus proving Theorem \ref{thm1}.
In Section \ref{sec:stabilization} where $\kappa$ is assumed to be positive, we will let $a:=\frac{\mu }{\kappa}$ and employ the functional
\[
 \calF_{\eps }(t)=\io \ue(\cdot,t)-\frac{\kappa}{\mu }\io \ln\ue(\cdot,t) + \frac{\kappa}{2\mu }\io \ve^2(\cdot,t)
\]
in order to derive the stabilization result in Theorem \ref{thm3} and already prepare Theorem \ref{thm:weaksol-limit}. Section \ref{sec:weaksol}, finally, will be devoted to the construction of weak solutions to \eqref{a}, and to the proofs of Theorem \ref{thm:weaksol} and Theorem \ref{thm:weaksol-limit}.

\begin{remark}
 In \eqref{epssys}, the additional term $-\eps \ue^2\ln a\ue$ could be replaced by $-\eps \Phi(\ue)$ with some other continuous function $\Phi$ which satisfies: $\Phi(s)\to 0$ as $s\searrow 0$, $\frac{\Phi(s)}{s^2}\to \infty$ as $s\to\infty$ and, for the stabilization results in Section \ref{sec:stabilization}, $\Phi<0$ on $(0,\frac{\kappa}{\mu})$ as well as $\Phi>0$ on $(\frac{\kappa}{\mu},\infty)$. \\
 We will always let
 \begin{equation}\label{defa}
  a:=\begin{cases} \frac{\mu}{\kappa},& \text{if }\; \kappa>0\\ \mu&\text{if }\; \kappa\leq 0\end{cases}
 \end{equation}
 and note that the choice for the case $\kappa\leq 0$ was arbitrary and that in Sections \ref{sec:bdclasssol} and \ref{sec:weaksol}, the precise value of $a$ plays no important role.
\end{remark}

\textbf{Notation.} For solutions of PDEs we will use $\Tmax$ to denote their maximal time of existence (cf. also Lemma \ref{criterion}). Throughout the article we fix $N\in\mathbb{N}$ and a bounded, smooth domain $\Om\sub\mathbb{R}^N$.

\section{General preliminaries}\label{sec:prelim}
In this section we provide some estimates that are valid for all suitably regular functions and not only for solutions of the PDE under consideration.

\begin{lem}\label{lem:elementary.estimates}
 a) For any $c\in C^2(\Om)$:
\begin{equation}\label{eq:Delta.Hessian}
|\Delta c|^2\leq N|D^2 c|^2 \quad \text{throughout } \Om.
\end{equation}
 b) There are $C>0$ and $k>0$ such that every positive $c\in C^2(\Ombar)$ fulfilling $\delny c=0$ on $\dOm$ satisfies
\begin{equation}\label{eq:lembdrytermvi}
  -2\io \frac{|\Lap c|^2}{c} +\io \frac{|\na c|^2\Lap c}{c^2} \leq -k \io c|D^2\ln c|^2 -k \io\frac{|\na c|^4}{c^3} + C\io c.
\end{equation}
\end{lem}
\begin{proof}
a) Straightforward calculations yield
 \begin{align*}
  \kl{\sum_{i=1}^N c_{x_ix_i}}^2=\sum_{i,j=1}^N c_{x_ix_i}c_{x_jx_j}\leq \sum_{i,j=1}^N \kl{\frac12 c_{x_ix_i}^2 + \frac12 c_{x_jx_j}^2} 
 = N\sum_{i=1}^N c_{x_ix_i}^2 \leq N\sum_{i,j=1}^N c_{x_ix_j}^2.
 \end{align*}
b) This is \cite[Lemma 2.7 vi)]{lankeit_fluid}.
\end{proof}

Let us now derive the following interpolation inequality on which we will rely in obtaining an estimate for $\io u^p+\io |\na v|^{2p}$ in Section \ref{sec:bdclasssol}
\begin{lem}\label{l_interpolation}
Let $q\in[1,\infty)$. Then for any $c\in C^2(\Ombar )$
satisfying $c\frac{\partial c}{\partial \nu}=0$ on $\partial
\Omega$, the inequality
\begin{eqnarray}\label{inter0}
\|\nabla c\|^{2q+2}_{L^{2q+2}(\Omega)} \le
2(4q^2+N)\|c\|^2_{L^{\infty}} \big\||\nabla
c|^{q-1}D^2c\big\|^{2}_{L^2(\Omega)}
\end{eqnarray}
holds, where $D^2c$ denotes the Hessian of $c$.
\end{lem}
\begin{proof} Since $c\frac{\partial c}{\partial \nu}=0$ on
$\partial \Omega$, an integration by parts yields
$$\|\nabla c\|^{2q+2}_{L^{2q+2}(\Omega)}=-\int_{\Omega}c|\nabla c|^{2q}\Delta c-2q\int_{\Omega}c |\nabla c|^{2q-2} \nabla c\cdot (D^2 c\cdot \nabla c).$$
Using Young's inequality and \eqref{eq:Delta.Hessian} we can estimate
\begin{eqnarray}\label{inter1}
\Big|-\int_{\Omega}c|\nabla c|^{2q}\Delta c\Big| &\le &
\frac{1}{4}\int_{\Omega}|\nabla
c|^{2q+2}+\int_{\Omega}c^2|\nabla
c|^{2q-2}|\Delta c|^2\nn\\
 &\le&
\frac{1}{4}\int_{\Omega}|\nabla
c|^{2q+2}+N\|c\|^2_{L^{\infty}(\Omega)}\int_{\Omega}|\nabla
c|^{2q-2}|D^2 c|^2.
\end{eqnarray}
Likewise, we see that
\begin{eqnarray}\label{inter2}
\Big|-2q\int_{\Omega}c |\nabla c|^{2q-2} \nabla c\cdot (D^2 c\cdot
\nabla c)\Big| &\le &
\frac{1}{4}\int_{\Omega}|\nabla
c|^{2q+2}+4q^2\|c\|^2_{L^{\infty}(\Omega)}\int_{\Omega}|\nabla
c|^{2q-2}|D^2 c|^2.
\end{eqnarray}
In consequence, \eqref{inter1} and \eqref{inter2} prove
\eqref{inter0}. \qquad
\end{proof}

\section{Local existence and basic properties of solutions}\label{sec:locex-andbasic}
We first recall a result on local solvability of \eqref{epssys}:
\begin{lem}\label{criterion}
Let $u_0$, $v_0$ satisfy \eqref{id}, let $\kappa\in\mathbb{R} $,
$\mu >0$, $\chi >0$ and $q>N$. Then for any $\eps
\in[0,1)$ there exist $T_{max}\in (0,\infty]$ and unique classical
solution $(\ue,\ve)$ of system \eqref{epssys} with $a$ as in \eqref{defa} in
$\Omega\times(0,T_{max})$ such that
\begin{eqnarray*}
&&\ue\in C^{0}\big(\Ombar \times[0,T_{\rm max})\big)\cap
C^{2,1}\big(\Ombar \times(0,T_{\rm max})\big),\\
&&\ve\in C^{0}\big(\Ombar \times[0,T_{\rm max})\big)\cap
C^{2,1}\big(\Ombar \times(0,T_{\rm max})\big).
\end{eqnarray*}
Moreover, we have $\ue> 0$ and $\ve > 0$ in $\Ombar \times
[0, T_{\max})$, and
\begin{equation}\label{extcrit}
\mathrm{if} \,\,\, T_{\rm max}<\infty,\,\,\, \mathrm{then}
\,\,\,\limsup_{t\nearrow\Tmax}\kl{\|\ue(\cdot,t)\|_{L^{\infty}(\Omega)}+\|\ve(\cdot,t)\|_{W^{1,q}(\Omega)}}=\infty.
\end{equation}
\end{lem}
\begin{proof}
 Apart from minor adaptions necessary if $\eps >0$ (see also \cite[Lemma 3.1]{win_ksns_logsource}), this lemma is contained in \cite[Lemma 2.1]{win_ctfluid}.
\end{proof}

Even thought the total mass is not conserved, an upper bound for it can be obtained easily:

\begin{lem}\label{lu1}
Let $u_0$, $v_0$ satisfy \eqref{id}, let $\kappa\in\mathbb{R} $, $\mu >0$, $\chi >0$. Then for any $\eps \in[0,\infty)$ the solution of \eqref{epssys} with $a$ as in \eqref{defa} satisfies
 \begin{eqnarray}\label{2.1}
 \iO \ue(x,t) dx\le \max\Big\{\f{\kappa|\Om|}{2\mu }+\sqrt{\kl{\f{\kappa_+|\Om|}{2\mu }}^2+\eps \f{|\Om|}{2a^2e\mu }},\io u_0 \Big\}=:m_{\eps } \quad\mathrm{ for\,\,\, all }\quad
t\in(0,T_{\rm max}).
\end{eqnarray}
\end{lem}
\begin{proof}
 Because $s^2\ln(as)\geq -\frac1{2a^2e}$ for all $s>0$, integrating the first equation in \eqref{epssys} over $\Omega$ and applying Hölder's inequality shows that
\[
 \ddt \io \ue\leq \kappa\io \ue - \frac{\mu }{|\Om|}\kl{\io \ue}^2+\frac{\eps |\Om|}{2a^2e}\quad \text{ on } (0,\Tmax)
\]
 and the claim results from an ODI-comparison argument.
\end{proof}

For the second component, even uniform boundedness can be deduced instantly:

\begin{lem}\label{lem:normvdecreases}
Let $u_0$, $v_0$ satisfy \eqref{id}, let $\kappa\in\mathbb{R} $, $\mu >0$, $\chi >0$.
Then for any $\eps \in[0,1)$ the solution of \eqref{epssys} with $a$ as in \eqref{defa} satisfies
 \begin{eqnarray}\label{2.2}
 \|\ve(\cdot,t)\|_{L^{\infty}(\Omega)}\le \|v_0\|_{L^{\infty}(\Omega)} \qquad\mathrm{ for\,\,\, all }\quad
t\in(0,T_{\rm max})
\end{eqnarray}
 and
\[
 (0,\Tmax)\ni t\mapsto \norm[\Lom\infty]{\ve(\cdot,t)}
\]
is monotone decreasing.
\end{lem}
\begin{proof}
 This is a consequence of the maximum principle and the nonnegativity of the solution.
\end{proof}

Also the gradient of $v$ can be controlled in an $L^2(\Om)$-sense:
\begin{lem}\label{ltv2} Let $u_0$, $v_0$ satisfy \eqref{id}, let $\kappa\in\mathbb{R} $, $\mu >0$, $\chi >0$.
There exists a positive constant $M$ such that for all $\eps \in[0,1)$ the solution of \eqref{epssys} with $a$ as in \eqref{defa} satisfies
\begin{eqnarray}\label{2.tv2}
\int_{\Omega}|\nabla
\ve(\cdot,t)|^2\le M
\qquad\mathrm{ for\,\,\, all }\quad t\in(0,T_{\rm max}).
\end{eqnarray}
\end{lem}
\begin{proof} Integration by parts and the Young inequality result
in
\begin{eqnarray}\label{2.tv1}
\frac{d}{dt}\int_{\Omega}|\nabla \ve|^2&=& 2\int_{\Omega} \nabla
\ve\cdot \nabla (\Delta \ve-\ue\ve)\nn \\
&\le& -2\int_{\Omega}|\Delta \ve|^2 -2\int_{\Omega}|\nabla
\ve|^2+2\int_{\Omega} \ve(\ue-1)\Delta \ve \nn\\
&\le& -\int_{\Omega}|\Delta \ve|^2 -2\int_{\Omega}|\nabla \ve|^2
+\int_{\Omega} \ve^2(\ue-1)^2\nn\\
&\le& -\int_{\Omega}|\Delta \ve|^2 -2\int_{\Omega}|\nabla
\ve|^2+\|v_0\|_{L^{\infty}(\Omega)}^2\int_{\Omega}
\ue^2+2\|v_0\|_{L^{\infty}(\Omega)}^2\int_{\Omega}
\ue+\|v_0\|_{L^{\infty}(\Omega)}^2
\end{eqnarray}
on $(0,\Tmax)$. Furthermore,
\begin{eqnarray}\label{2.tv3}
\frac{\|v_0\|_{L^{\infty}(\Omega)}^2}{\mu}\frac{d}{dt}\int_{\Omega}
\ue\le \frac{\kappa_{+}\|v_0\|_{L^{\infty}(\Omega)}^2}{\mu}\int_{\Omega}
\ue-\|v_0\|_{L^{\infty}(\Omega)}^2\int_{\Omega} \ue^2-\frac{\eps \norm[\Lom\infty]{v_0}^2}{\mu }\io \ue^2\ln (a\ue).
\end{eqnarray}
Adding \eqref{2.tv1} to \eqref{2.tv3} and taking into account that
\[
 -\frac{\eps \norm[\Lom\infty]{v_0}^2}{\mu } s^2\ln  as\leq \frac{\norm[\Lom\infty]{v_0}^2}{2ea^2\mu }
\]
for any $\eps \in[0,1)$ and $s\geq 0$, we obtain that
\begin{eqnarray*}
&&\frac{d}{dt}\Big\{\int_{\Omega}|\nabla
\ve|^2+\frac{\|v_0\|_{L^{\infty}(\Omega)}^2}{\mu}\int_{\Omega}\ue\Big\}\\
&& \le -\left(\int_{\Omega}|\nabla
\ve|^2+\frac{\|v_0\|_{L^{\infty}(\Omega)}^2}{\mu}\int_{\Omega}\ue\right)+\left(2\|v_0\|_{L^{\infty}(\Omega)}^2+\frac{\kappa_{+}+1}{\mu}\|v_0\|_{L^{\infty}(\Omega)}^2\right)\int_{\Omega}
\ue+\|v_0\|_{L^{\infty}(\Omega)}^2+\frac{|\Om|\norm[\Lom\infty]{v_0}^2}{2\mu ea^2}.
\end{eqnarray*}
Since Lemma \ref{lu1} shows that $\iO \ue(x,t) dx\le m_1$ for any $\eps \in[0,1)$ and $t\in(0,\Tmax)$,
a comparison argument leads to
\begin{eqnarray*}
&&\int_{\Omega}|\nabla
\ve|^2+\frac{\|v_0\|_{L^{\infty}(\Omega)}^2}{\mu}\int_{\Omega}\ue \\
&& \le \max \left\{|\nabla
v_0|^2+\frac{\|v_0\|_{L^{\infty}(\Omega)}^2}{\mu}\int_{\Omega}u_0,\,\kl{1+\frac{|\Om|}{2ea^2\mu }}\|v_0\|_{L^{\infty}(\Omega)}^2+
\left(2\|v_0\|_{L^{\infty}(\Omega)}^2+\frac{\kappa_{+}+1}{\mu}\|v_0\|_{L^{\infty}(\Omega)}^2\right)m_1\right\},
\end{eqnarray*}
holding true on $(0,\Tmax)$, which in particular implies \eqref{2.tv2}
\end{proof}

\section{Existence of a bounded classical solution}\label{sec:bdclasssol}

We now turn to the analysis of the coupled functional of
$\int_{\Omega}u^p$ and $\int_{\Omega}|\nabla v|^{2p}$. We first
apply standard testing procedures to gain the time evolution of
each quantity.

\begin{lem}\label{l3.1} Let $u_0$, $v_0$ satisfy \eqref{id}, let $\kappa\in\mathbb{R} $, $\mu >0$, $\chi >0$. For any $p\in[1,\infty)$, any $\eps \in[0,1)$, we have that the solution of \eqref{epssys} with $a$ as in \eqref{defa} satisfies
\begin{equation}\label{3.1}
\frac{d}{dt}\int_{\Omega}\ue^p+\frac{2(p-1)}{p}\int_{\Omega}|\nabla
\ue^{\frac{p}{2}}|^2\le \frac{p(p-1)}{2} \chi
^2\int_{\Omega}\ue^p|\nabla \ve|^2+p\kappa\int_{\Omega}\ue^p-p\mu
\int_{\Omega}\ue^{p+1} -\eps  p\io \ue^{p+1}\ln  a\ue
\end{equation}
on $(0,\Tmax)$.
\end{lem}
\begin{proof}  Testing the first equation in \eqref{a} against
$\ue^{p-1}$ and using Young's inequality, we can obtain
\begin{align}\label{3.1''}
\frac{1}{p}\frac{d}{dt}\int_{\Omega}\ue^p
&=-(p-1)\int_{\Omega}\ue^{p-2}|\nabla
\ue|^2+(p-1)\chi \int_{\Omega}\ue^{p-1}\nabla \ue \cdot \nabla
\ve+\kappa\int_{\Omega}\ue^p-\mu\int_{\Omega}\ue^{p+1}\nn\\
&\quad -\eps \io \ue^{p+1}\ln (a\ue)\nn\\
&\le -(p-1)\int_{\Omega}\ue^{p-2}|\nabla
\ue|^2+\frac{p-1}{2}\int_{\Omega}\ue^{p-2}|\nabla
\ue|^2+\frac{p-1}{2} \chi^2\int_{\Omega}\ue^{p}|\nabla
\ve|^2
\nn\\
&\quad+\kappa\int_{\Omega}\ue^p-\mu\int_{\Omega}\ue^{p+1}-\eps \io \ue^{p+1}\ln (a\ue)
\end{align}
on $(0,T_{\rm max})$, which  by using the fact that
$$\int_{\Omega}u^{p-2}|\nabla
\ue|^2=\frac{4}{p^2}\int_{\Omega}|\nabla \ue^{\frac{p}{2}}|^2 \quad \text{ on } (0,\Tmax)$$
directly results in
\eqref{3.1}.
\end{proof}

\begin{lem}\label{l3.2} Let $u_0$, $v_0$ satisfy \eqref{id}, let $\kappa\in\mathbb{R} $, $\mu >0$, $\chi >0$.
 For any $p\in[1,\infty)$, any $\eps \in[0,1)$, we have that the solution of \eqref{epssys} with $a$ as in \eqref{defa} satisfies
\begin{eqnarray}\label{3.1'}
\frac{d}{dt}\int_{\Omega} |\nabla \ve|^{2p}+p\int_{\Omega}|\nabla
\ve|^{2p-2}|D^2 \ve|^2 \le
p(p+N-1)\|v_0\|_{L^{\infty}(\Omega)}^2\int_{\Omega}\ue^2|\nabla
\ve|^{2p-2} \quad \text{on } (0,T_{\rm
max}).
\end{eqnarray}
\end{lem}
\begin{proof} We differentiate the second equation in \eqref{a} to
compute
$$(|\nabla \ve|^2)_t=2\nabla \ve\cdot \nabla \Delta \ve-2\nabla \ve\cdot \nabla (\ue\ve)=\Delta |\nabla \ve|^2-2|D^2\ve|^2-2\nabla \ve\cdot \nabla (\ue\ve)
\quad \text{in }  \Omega\times (0,T_{\rm max}).$$ Upon multiplication
by $(|\nabla \ve|^2)^{p-1}$ and integration, this leads to
\begin{equation}\label{3.2}
\frac{1}{p}\frac{d}{dt} \int_{\Omega}|\nabla
\ve|^{2p}+(p-1)\int_{\Omega} |\nabla \ve|^{2p-4}\big|\nabla|\nabla
\ve|^2\big|^2+2\int_{\Omega}|\nabla \ve|^{2p-2}|D^2\ve|^2\le
-2\int_{\Omega}|\nabla \ve|^{2p-2} \nabla \ve \cdot \nabla(\ue\ve)
\end{equation}
on $(0,T_{\rm max})$. Then integrating by parts, we
achieve
\begin{align*}
 -2\int_{\Omega}|\nabla \ve|^{2p-2} \nabla \ve \cdot
\nabla(\ue\ve)&=2\int_{\Omega} \ue\ve|\nabla \ve|^{2p-2} \Delta
\ve+2(p-1)\int_{\Omega} \ue\ve|\nabla \ve|^{2p-4}\nabla \ve\cdot
\nabla|\nabla \ve|^2\\
&\le 2\|v_0\|_{L^{\infty}(\Omega)}\int_{\Omega} \ue|\nabla
\ve|^{2p-2} |\Delta
\ve|+2(p-1)\|v_0\|_{L^{\infty}(\Omega)}\int_{\Omega} \ue|\nabla
\ve|^{2p-3}\cdot \big|\nabla|\nabla \ve|^2\big|
\end{align*}
 throughout $(0,\Tmax)$, were we have used Lemma \ref{lem:normvdecreases}.
 Next by Young's inequality and Lemma \ref{lem:elementary.estimates} a)
we have that
 $$2\|v_0\|_{L^{\infty}(\Omega)}\int_{\Omega} \ue|\nabla
\ve|^{2p-2} |\Delta \ve|\le \int_{\Omega} |\nabla \ve|^{2p-2} |D^2
\ve|^2+N\|v_0\|^2_{L^{\infty}(\Omega)}\int_{\Omega}\ue^2 |\nabla
\ve|^{2p-2},$$

and
\begin{eqnarray*}
&&  2(p-1)\|v_0\|_{L^{\infty}(\Omega)}\int_{\Omega} \ue|\nabla
\ve|^{2p-3}\cdot \big|\nabla|\nabla \ve|^2\big|\\
&&\le (p-1)\int_{\Omega} |\nabla \ve|^{2p-4}\big|\nabla|\nabla
\ve|^2\big|^2+(p-1)\|v_0\|^2_{L^{\infty}(\Omega)}\int_{\Omega}\ue^2
|\nabla \ve|^{2p-2}.
\end{eqnarray*}

Thereupon, \eqref{3.2} implies that
$$\frac{d}{dt} \int_{\Omega}|\nabla
\ve|^{2p}+p\int_{\Omega}|\nabla \ve|^{2p-2}|D^2\ve|^2\le
p(p+N-1)\|v_0\|^2_{L^{\infty}(\Omega)}\int_{\Omega}\ue^2 |\nabla
\ve|^{2p-2}$$ on $(0,T_{\rm max})$.
\end{proof}

Next we will show that if $\mu$ is suitably large, then all
integrals on the right side in \eqref{3.1} and \eqref{3.1'}
can adequately be estimated in terms of the respective dissipated
quantities on the left, in consequence implying the $L^p$ estimate
of $u$ and the boundedness estimate for $|\nabla v|$.

\begin{lem}\label{l3.5}
Let $p>1$. With
\begin{align}\label{eq:defk1k2}
 k_1(p,N):=\f{p(p-1)}{(p+1)}\left(\frac{4(p-1)(4p^2+N)}{p+1}\right)^{\frac{1}{p}}\nn\\
 k_2(p,N):=\frac{4(p+N-1)}{p+1}\left(\frac{8(p-1)(p+N-1)(4p^2+N)}{p+1}\right)^{\frac{p-1}{2}}
\end{align}
the following holds:
If $\mu >0$, $\chi >0$ and the positive function $v_0\in C^1(\Ombar)$ fulfil
\begin{equation}\label{cond:mularge}
\mu\ge k_1(p,N)\|\chi v_0\|^{\frac{2}{p}}_{L^\infty(\Omega)}+k_2(p,N)\|\chi v_0\|^{2p}_{L^\infty(\Omega)},
\end{equation}
 then
for every $\kappa\in \mathbb{R} $, $0<u_0\in C^0(\Ombar)$ there is $C>0$ such that for every $\eps \in[0,1)$ the solution $(\ue,\ve)$ of \eqref{epssys} with $a$ as in \eqref{defa} satisfies
\[
 \io \ue^p(\cdot,t) + \io |\na \ve(\cdot,t)|^{2p}\leq C \qquad \text{on } (0,\Tmax).
\]
If, however, $\mu >0$, $\chi >0$ and $0<v_0\in C^1(\Ombar)$ do not satisfy \eqref{cond:mularge}, then
for every $\eps \in(0,1)$, $\kappa\in \mathbb{R} $, $0<u_0\in C^0(\Ombar)$ there is $c_\eps>0$ such that the solution $(\ue,\ve)$ of \eqref{epssys} with $a$ as in \eqref{defa} satisfies
\[
 \io \ue^p(\cdot,t) + \io |\na \ve(\cdot,t)|^{2p}\leq c_\eps \qquad \text{on } (0,\Tmax).
\]
\end{lem}
\begin{proof} Lemma \ref{l3.1} and \ref{l3.2} show that
\begin{align}\label{3.12}
& \frac{d}{dt}\Big(\int_{\Omega}\ue^p+\chi ^{2p}\int_{\Omega}|\nabla
\ve|^{2p}\Big)+\frac{2(p-1)}{p}\int_{\Omega}|\nabla
\ue^{\frac{p}{2}}|^2+p\chi ^{2p}\int_{\Omega}|\nabla \ve|^{2p-2}|D^2 \ve|^2 \nn\\
&\le \frac{p(p-1)}{2} \chi ^2\int_{\Omega}\ue^p|\nabla
\ve|^2+p(p+N-1)\|v_0\|_{L^{\infty}(\Omega)}^2\chi ^{2p}\int_{\Omega}\ue^2|\nabla
\ve|^{2p-2}\nn\\
&+p\kappa\int_{\Omega}\ue^p-p\mu \int_{\Omega}\ue^{p+1}- \eps p\io \ue^{p+1}\ln  a\ue
\end{align}
throughout $(0,T_{\rm{max}})$. Using Young's inequality, we can assert that for any
$\delta_1>0$,
\begin{eqnarray}\label{3.13}
 \frac{p(p-1)}{2} \chi ^{2}\int_{\Omega}\ue^p|\nabla
\ve|^2\le \frac{p(p-1)\delta_1 ^{p+1}}{2(p+1)}\chi ^{2p}\int_{\Omega}|\nabla
\ve|^{2(p+1)}+\frac{p^2(p-1)}{2(p+1)}\Big(\frac{1}{\delta_1}\Big)^{\frac{p+1}{p}}\chi ^{\frac2p}\int_{\Omega}\ue^{p+1}
\end{eqnarray}
on $(0,T_{\rm max})$. We then apply Lemma
\ref{l_interpolation} and $\|\ve(\cdot,t)\|_{L^{\infty}(\Omega)}\le
\|v_0\|_{L^{\infty}(\Omega)}$ to obtain
$$ \frac{p(p-1)\delta_1 ^{p+1}}{2(p+1)}\chi ^{2p}\int_{\Omega}|\nabla
\ve|^{2(p+1)}\le \frac{p(p-1)(4p^2+N)\delta_1
^{p+1}\|v_0\|^2_{L^\infty(\Omega)}}{(p+1)}\chi ^{2p}\int_{\Omega}|\nabla
\ve|^{2p-2}|D^2\ve|^2 $$ for all $t\in (0,T_{\rm max})$. If we let
$\delta_1=\left(\frac{p+1}{4(p-1)(4p^2+N)\|v_0\|^2_{L^\infty(\Omega)}}\right)^{\frac{1}{p+1}}$,
\eqref{3.13} shows that
\begin{equation}\label{3.14}
\frac{p(p-1)}{2}\chi ^2 \int_{\Omega}\ue^p|\nabla \ve|^2\le
\frac{p}{4}\chi ^{2p}\int_{\Omega}|\nabla
\ve|^{2p-2}|D^2\ve|^2+\frac{p^2(p-1)}{2(p+1)}\left(\frac{4(p-1)(4p^2+N)}{p+1}\right)^{\frac{1}{p}}\|v_0\|^{\frac{2}{p}}_{L^\infty(\Omega)}\chi ^{\frac2p}\int_{\Omega}\ue^{p+1}
\end{equation}
on $(0,T_{\rm max})$. Similarly, for any $\delta_2>0$ we have
\begin{eqnarray}\label{3.15}
&& p(p+N-1)\|v_0\|_{L^{\infty}(\Omega)}^2\chi ^{2p}\int_{\Omega}u^2|\nabla
\ve|^{2p-2} \nn\\
&& \le \frac{p(p-1)(p+N-1)\delta_2
^{\frac{p+1}{p-1}}\|v_0\|^2_{L^\infty(\Omega)}}{p+1}\chi
^{2p}\int_{\Omega}|\nabla
\ve|^{2(p+1)}\nn\\
&&\quad
+\frac{2p(p+N-1)\|v_0\|^2_{L^\infty(\Omega)}}{p+1}\Big(\frac{1}{\delta_2}\Big)^{\frac{p+1}{2}}\chi
^{2p}\int_{\Omega}\ue^{p+1}
\end{eqnarray}
on $(0,T_{\rm max})$. Using Lemma \ref{l_interpolation} once more and taking
$\delta_2=\left(\frac{p+1}{8(p-1)(p+N-1)(4p^2+N)\|v_0\|^4_{L^\infty(\Omega)}}\right)^{\frac{p-1}{p+1}}$,
we can obtain from \eqref{3.15} that
\begin{eqnarray}\label{3.17}
&& p(p+N-1)\|v_0\|_{L^{\infty}(\Omega)}^2\chi ^{2p}\int_{\Omega}\ue^2|\nabla
\ve|^{2p-2}\nn\\
&& \le \frac{p}{4}\chi ^{2p}\int_{\Omega}|\nabla
\ve|^{2p-2}|D^2\ve|^2\nn\\
&&\quad+\frac{2p(p+N-1)}{p+1}\left(\frac{8(p-1)(p+N-1)(4p^2+N)}{p+1}\right)^{\frac{p-1}{2}}\|v_0\|^{2p}_{L^\infty(\Omega)}\chi
^{2p}\int_{\Omega}\ue^{p+1}
\end{eqnarray}
on $(0,T_{\rm max})$.
Combining inequalities \eqref{3.12}, \eqref{3.13}
and \eqref{3.17}, we arrive at
\begin{align}\label{eq:diffineq}
&\frac{d}{dt}\kl{\int_{\Omega}\ue^p+\chi ^{2p}\int_{\Omega}|\nabla
\ve|^{2p}}+\frac{2(p-1)}{p}\int_{\Omega}|\nabla
\ue^{\frac{p}{2}}|^2+\frac{p}{2}\chi ^{2p}\int_{\Omega}|\nabla
\ve|^{2p-2}|D^2 \ve|^2 \\\nn &\le
\frac{p}2\kl{k_1(p,N)\norm[\Lom\infty]{\chi v_0}^{\frac2p} +
k_2(p,N)\norm[\Lom\infty]{\chi v_0}^{2p} - \mu }\io
\ue^{p+1} - \eps \io \ue^{p+1}\ln a\ue + p\kappa\io \ue^p -
\frac{\mu p}2\io \ue^{p+1}
\end{align}
on $(0,T_{\rm{max}})$.

We can moreover invoke the Poincar\'{e} inequality along with Lemma \ref{lu1} to
estimate
\begin{eqnarray*}
\int_{\Omega}\ue^p = \|\ue^{\frac{p}{2}}\|^2_{L^2(\Omega)}\le
c_1\big(\|\nabla
\ue^{\frac{p}{2}}\|^2_{L^2(\Omega)}+\|\ue^{\frac{p}{2}}\|^2_{L^{\frac{2}{p}}(\Omega)}\big)\le
c_2\Big(\int_{\Omega}|\nabla \ue^{\frac{p}{2}}|^2+1\Big) \qquad
\text{on } (0,T_{\rm max})
\end{eqnarray*}
with some $c_1>0$ and $c_2>0$. In a quite similar way, using Lemma
\ref{ltv2} we obtain constants $c_3>0$ and $c_4>0$ such that
\begin{eqnarray*}
\int_{\Omega}|\nabla \ve|^{2p} &=& \big\||\nabla
\ve|^{p}\big\|^2_{L^2(\Omega)}\\
&\le& c_3\big(\big\|\nabla |\nabla \ve|^{p}
\big\|^2_{L^2(\Omega)}+\big\||\nabla
\ve|^{p}\big\|^2_{L^{\frac{2}{p}}(\Omega)}\big)\\
&\le& c_4\Big(\int_{\Omega}\big|\nabla |\nabla
\ve|^{p}\big|^2+1\Big)\\
&=& c_4\Big(p^2\int_{\Omega}|\nabla \ve|^{2p-4}|D^2\ve \nabla
\ve|^2+1\Big)\\
&\le& c_4\Big(p^2\int_{\Omega}|\nabla \ve|^{2p-2}|D^2\ve |^2+1\Big)
 \qquad \text{on }
(0,T_{\rm max}).
\end{eqnarray*}
Introducing $c_5:=\min\set{\frac{2(p-1)}{c_2p},\frac{p}{2c_4}}$ and abbreviating $\ye(t):=\io \ue^p+\io |\na \ve|^{2p}$, we thus obtain from \eqref{eq:diffineq} that
\[
 \ye'(t)+c_5 \ye(t) \leq K \qquad \text{for all } t\in(0,\Tmax),
\]
where
\[
 K:=\begin{cases}
     p|\Om|\cdot\kl{\sup_{s>0} (\kappa s^p -\frac{\mu }2s^{p+1}) + |\inf_{s>0} s^{p+1} \ln s|}, &\text{ if \eqref{cond:mularge}}\\
      p|\Om| \sup_{s>0} \kl{\kl{\frac12 k_1(p,N)\norm[\Lom\infty]{v_0}^{\frac2p} + \frac12 k_2(p,N)\norm[\Lom\infty]{v_0}^{2p}-\mu } s^{p+1} + \kappa s^p -\eps s^{p+1}\ln  as}&\text{ else}.
    \end{cases}
\]
In consequence,
\[
 \ye(t)\leq \max\set{\ye(0);\frac{K}{c_{42}}}
\]
for all $t\in(0,\Tmax)$. We note that $K$ depends on $\eps $ if and only if \eqref{cond:mularge} is not satisfied.
\end{proof}

The previous lemma ensures boundedness of $u$ in some $L^p$-space for finite $p$ only. Fortunately, this is already sufficient for the solution to be bounded -- and global.

\begin{lem}\label{lem:ulp.to.boundedness}
 Let $T\in(0,\infty]$, $p>N$, $M>0$, $a>0$, $\kappa\in\mathbb{R} $, $\mu >0$. Then there is $C>0$ with the following property:
 If for some $\eps \in[0,1)$, the function $(\ue,\ve)\in (C^0(\Ombar\times[0,T))\cap C^{2,1}(\Ombar\times(0,T)))^2$ is a solution to \eqref{epssys} with $a$ as in \eqref{defa} such that
\[
 0\leq \ue,\,\, 0\leq \ve \text{ in  }\Om\times(0,T)\text{ and } \io \ue^p(\cdot,t)\leq M \text{ for all } t\in (0,T),
\]
then
\[
 \norm[\Lom\infty]{\ue(\cdot,t)} + \norm[\Lom\infty]{\na \ve(\cdot,t)}\leq C \qquad \text{for all } t\in(0,T).
\]
\end{lem}
\begin{proof}
We use the standard estimate for the Neumann heat
semigroup (\cite[Lemma 1.3]{win_aggregationvs}) to conclude that with some $c_1>0$
\begin{align*}
\|\nabla \ve(\cdot,t)\|_{L^{\infty}(\Omega)}&\le \|\nabla
e^{t\triangle} \ve(\cdot,0)\|_{L^{\infty}(\Omega)}+\int^t_0 \|\nabla
e^{(t-s)\triangle}\ue(\cdot,s)\ve(\cdot,s)\|_{L^{\infty}(\Omega)}\\
&\le \norm[\Lom\infty]{\na v_0}+c_1\int^t_0
c_1\Big(1+(t-s)^{-\f{1}{2}-\f N{2p}}\Big)e^{-\lambda_1(t-s)}\|\ue(\cdot,s)\ve(\cdot,s)\|_{L^{p}(\Omega)} \quad \text{for all } t\in(0,T),
\end{align*}
where $\lambda_1$ denotes the first nonzero eigenvalue of
$-\Delta$ in $\Omega$  under the homogeneous Neumann boundary
conditions. Due to Lemma \ref{lem:normvdecreases} and the condition on $\ue$, we obtain $c_2>0$ such that
\begin{equation}\label{vlinftybound}
 \norm[\Lom \infty]{\na \ve(\cdot,t)}\leq c_2 \quad \text{for all } t\in (0,T).
\end{equation}
In order to obtain a bound for $\ue$, we use the
variation-of-constants formula to represent $\ue(\cdot,t)$ as
\begin{eqnarray}\label{duhamel.u}
\ue(\cdot,t)&=&e^{(t-t_0)\Delta}\ue(\cdot,t_0)-\int^t_{t_0}
e^{(t-s)\Delta}\nabla\cdot(\ue(\cdot,s)\nabla \ve(\cdot,s))ds\nn\\
&\quad&+\int^t_{t_0}e^{(t-s)\Delta}(\kappa \ue(\cdot,s)-\mu
\ue^2(\cdot,s)-\eps \ue^2(\cdot,s)\ln  a\ue(\cdot,s))ds,
\end{eqnarray}
for each $t\in (0,T)$, where $t_0=(t-1)_{+}$.

Due to the estimate
\[
 \kappa s-\mu s^2-\eps s^2\ln  as\leq \frac1{2a^2e} + \sup_{\xi >0} (\kappa\xi -\mu \xi ^2)=:c_3,
\]
positivity of the heat semigroup ensures that
\begin{equation}\label{uestimate:inhomogeneity}
 \int_{t_0}^t e^{(t-s)\Delta }(\kappa\ue(\cdot,s)-\mu \ue^2(\cdot,s)-\eps \ue^2(\cdot,s)\ln  a\ue(\cdot,s)) \leq c_3(t-t_0)\leq c_3.
\end{equation}
Moreover, from the maximum principle we can easily infer that
\begin{equation}\label{uestimate:initdata.smallt}
 \norm[\Lom\infty]{e^{(t-t_0)\Delta} \ue(\cdot,t_0)}=\norm[\Lom\infty]{e^{t\Delta} u_0} \leq \norm[\Lom\infty]{u_0} \text{ if } t\in[0,2]
\end{equation}
and that with $c_4>0$ taken from \cite[Lem. 1.3]{win_aggregationvs}
\begin{equation}\label{uestimate:initdata.larget}
 \norm[\Lom\infty]{e^{(t-t_0)\Delta}\ue(\cdot,t_0)}\leq \norm[\Lom\infty]{\ue(\cdot,t_0)}\leq \norm[\Lom\infty]{e^{1\cdot\Delta} \ue(\cdot,t_0-1)} \leq c_4 (1+1^{-\frac N2})
  \norm[\Lom 1]{\ue(\cdot,t_0-1)}\leq c_4 m_1,
\end{equation}
whenever $t>2$ and with $m_1$ as in \eqref{2.1}.

Finally, we estimate the second integral on the right hand of
\eqref{duhamel.u}.  \cite[Lemma 1.3]{win_aggregationvs} provides $c_5>0$ fulfilling
\begin{align}\label{uestimate:chemotaxisterm}
\Big\|\int^t_{t_0} e^{(t-s)\Delta}\nabla\cdot(\ue(\cdot,s)\nabla
\ve(\cdot,s))ds\Big\|_{L^{\infty}(\Omega)}&\le c_5\int^t_{t_0}
(t-s)^{-\frac{1}{2}-\frac{N}{2p}}\|\ue(\cdot,s)\nabla
\ve(\cdot,s)\|_{L^{p}(\Omega)}ds\nn\\
&\leq c_5M c_2\int_0^1\sigma^{-\frac12-\frac N{2p}} d\sigma
=:c_6
\end{align}
for $t\in(0,T)$. In view of \eqref{duhamel.u}, \eqref{uestimate:initdata.smallt}, \eqref{uestimate:initdata.larget}, \eqref{uestimate:chemotaxisterm}, we have obtained that
\[
 0\leq \ue(\cdot,t)\leq \max\set{\norm[\Lom\infty]{u_0}, c_4m_1} + c_6 + c_3
\]
holds for any $t\in(0,T)$, which combined with \eqref{vlinftybound} is the desired conclusion.
\end{proof}

In fact, the assumption of Lemma \ref{lem:ulp.to.boundedness} suffices for even higher regularity, as we will see in Lemma \ref{lem:holder}.
For the moment we return to the proof of global existence of solutions.

\begin{lem}\label{lem:ge.for.positive.eps.or.large.mu}
 Let $\eps \in(0,1)$ and let $a$ be as in \eqref{defa} or let $\eps =0$ and $\mu >k_1(N,N)\norm[\Lom\infty]{\chi v_0}^{\frac2N}+k_2(N,N)\norm[\Lom\infty]{\chi v_0}^{2N}$, where $k_1$, $k_2$ are as in Lemma \ref{l3.5}. Then the classical solution to \eqref{epssys} given by Lemma \ref{criterion} is global and bounded.
\end{lem}
\begin{proof}
 By continuity, there is $p>N$ such that $\mu >k_1(p,N)\norm[\Lom\infty]{\chi v_0}^{\frac2p}+k_2(p,N)\norm[\Lom\infty]{\chi v_0}^{2p}$, and Lemma \ref{l3.5} shows that $\io \ue^p$ is bounded on $(0,\Tmax)$. Lemma \ref{lem:ulp.to.boundedness} together with Lemma \ref{lem:normvdecreases} turns this into a uniform bound on $\norm[\Lom \infty]{\ue(\cdot,t)}+\norm[W^{1,\infty}(\Om)]{\ve(\cdot,t)}$ on $(0,\Tmax)$, so that the extensibility criterion \eqref{extcrit} shows that $\Tmax=\infty$.
\end{proof}

\begin{proof}[Proof of Theorem \ref{thm1}] Theorem \ref{thm1} is the case $\eps =0$ in Lemma \ref{lem:ge.for.positive.eps.or.large.mu}.\end{proof}

\section{Stabilization}\label{sec:stabilization}

In this section, we shall consider the large time asymptotic
stabilization of any global classical bounded solution.

In a first step we derive uniform Hölder bounds that will facilitate convergence. After that, we have to ensure that solutions actually converge, and in particular must identify their limit.
In the spirit of the persistence-of-mass result in \cite{taowin_persistence}, showing that $v\to 0$ as $t\to \infty$ would be possible by relying on a uniform lower bound for $\io u$ and finiteness of $\intninf\io uv$ (see also \cite[Lemmata 3.2 and 3.3]{lankeit_fluid}). We will instead focus on other information that can be obtained from the following functional of type already employed in \cite{lankeit_fluid} (after the example of \cite{win_ksns_logsource}), namely
\begin{equation}\label{eq:defF}
 \calF_{\eps}(t):=\int_{\Omega}\ue(\cdot,t)-\f{\kappa}{\mu }\int_{\Omega}\ln \ue(\cdot,t)+\f{\kappa}{2\mu }\int_{\Omega}\ve^2(\cdot,t).
\end{equation}
This way, in Lemma \ref{lem:vtozero} we will achieve a convergence result for $\ve$ that will also be useful in the investigation of the large time behaviour of weak solutions in Section \ref{sec:weaksol}.

\begin{lem}\label{lem:holder}
 Let $\eps\in[0,1)$, $\mu >0$, $\chi >0$, $\kappa\in\mathbb{R} $. Let $(\ue,\ve)\in C^0(\Ombar\times [0,\infty))\cap C^{2,1}(\Ombar\times(0,\infty))$ be a solution to \eqref{epssys} with $a$ as in \eqref{defa} which is bounded in the sense that there exists $M>0$ such that
 \begin{equation}\label{eq:boundednessconditionforregularity}
  \norm[\Lom\infty]{\ue(\cdot,t)}+\norm[W^{1,\infty}(\Om)]{\ve(\cdot,t)}\leq M \quad \text{for all } t\in(0,\infty).
 \end{equation}
 Then there are $\alpha\in(0,1)$ and $C>0$ such that
 \[
  \normm{C^{\alpha,\frac{\alpha}2}(\Ombar\times[t,t+1])}{\ue} + \normm{C^{2+\alpha,1+\frac\alpha2}(\Om\times[t,t+1])}{\ve}\leq C \qquad \text{for all } t\in (2,\infty).
 \]
\end{lem}
\begin{proof}
Due to the time-uniform ($L^\infty(\Om)$-)bound on $\ve$ and on the right hand side of
\[
 \vet-\Delta \ve = \ue\ve \text{ in } \Om\times[t,t+2], \quad \delny \ve\amrand=0
\]
of which $\ve$ is a weak solution, \cite[Thm. 1.3]{porzio_vespri} immediately yields $\alpha _1\in(0,1)$ and $c_1>0$ such that
\[
 \normm{C^{\alpha _1,\frac{\alpha _1}2}(\Ombar\times[t+1,t+2])}{\ve}\leq c_1
\]
for any $t>0$.

Similarly, \eqref{eq:boundednessconditionforregularity} provides $t$-independent bounds on the functions $\psi _0:=\frac12\chi ^2\ue^2|\na \ve|^2$, $\psi _1:=\chi \ue|\na \ve|$, $\psi _2:=|\kappa|\ue-\mu \ue^2-\eps \ue^2\ln  a\ue$ in conditions (A$_1$), (A$_2$), (A$_3$) of \cite{porzio_vespri}. An application of \cite[Thm. 1.3]{porzio_vespri} to solutions of
\[
 \uet - \na\cdot(\na \ue-\chi \ue\na \ve) = \kappa\ue-\mu \ue^2-\eps \ue^2\ln  a\ue \text{ in } \Om\times[t,t+2], \delny \ue\amrand=0
\]
therefore provides $\alpha _2\in(0,1)$, $c_2>0$ such that
\[
 \normm{C^{\alpha _2,\frac{\alpha _2}2}(\Ombar\times[t+1,t+2])}{\ue}\leq c_2
\]
for any $t>0$.

We pick a monotone increasing function $\zeta \in C^\infty(\mathbb{R} )$ such that $\zeta |_{(-\infty,\frac12)}\equiv 0$, $\zeta|_{(1,\infty)}\equiv 1$ and note that, for any $t_0>1$, the function $(x,t)\mapsto \zeta (t-t_0)\ve(x,t)$ belongs to $C^{2,1}(\Ombar\times[t_0,t_0+2])$ and satisfies
\[
 (\zeta \ve)_t=\Delta (\zeta \ve) - \ue\zeta \ve + \zeta '\ve, \quad (\zeta \ve)(\cdot,t_0)=0, \quad \delny(\zeta \ve)\amrand=0.
\]
Due to the uniform bound for $\ue\zeta \ve + \zeta '\ve$ in some Hölder space, an application of \cite[Thm. IV.5.3]{LSU} (together with \cite[Thm. III.5.1]{LSU}) ensures the existence of $\alpha _3\in(0,1)$ and $c_3>0$ such that
\[
 \normm{C^{2+\alpha _3,1+\frac{\alpha _3}2}(\Ombar\times[t_0+1,t_0+2])}{\ve}=\normm{C^{2+\alpha _3,1+\frac{\alpha _3}2}(\Ombar\times[t_0+1,t_0+2]}{\zeta \ve}\leq\normm{C^{2+\alpha _3,1+\frac{\alpha _3}2}(\Ombar\times[t_0,t_0+2]}{\zeta \ve}\leq c_3
\]
for any $t_0>1$.
\end{proof}

\begin{lem}\label{lem:ddtF}
Let $u_0$, $v_0$ satisfy \eqref{id} and assume that $\mu >0$, $\kappa>0$, $\chi >0$, $a=\f{\mu }{\kappa}$ (as in \eqref{defa}). Then for any $\eps \in[0,1)$ any solution $(\ue,\ve)\in C^{2,1}(\Ombar\times(0,\infty))\cap C^0(\Ombar\times[0,\infty))$ of \eqref{epssys} satisfies
\begin{equation}\label{FODI}
 \calF_\eps'(t)+\mu \io\kl{\ue-\frac{\kappa}{\mu }}^2\leq 0 \qquad \text{for all } t\in(0,\infty)
\end{equation}
and, consequently, there is $C>0$ such that for any $\eps \in [0,1)$
\begin{equation}\label{eq:ueminlimit2.bounded}
 \intninf \io \kl{\ue-\frac{\kappa}{\mu }}^2 \leq C.
\end{equation}

\end{lem}
 \begin{proof}
In fact, on $(0,\infty)$
\begin{align*}
\calF_\eps'&=\int_{\Omega}
\uet-\f{\kappa}{\mu }\int_{\Omega}\frac{\uet}{\ue}+\f{\kappa}{\mu }\int_{\Omega}\ve\vet
\nn \\
&= \kappa\int_{\Omega}
\ue-\mu\int_{\Omega}\ue^2 - \eps \io \ue^2\ln  a\ue -\f{\kappa}{\mu }\int_{\Omega}\frac{\Delta
\ue}{\ue}+\f{\kappa}{\mu }\int_{\Omega}\frac{\nabla \ue\cdot \nabla
\ve}{\ue}-\f{\kappa^2}{\mu }\int_{\Omega} 1+\kappa\int_{\Omega}
\ue\\
&+\f{\kappa}{\mu }\int_{\Omega}\ve\Delta
\ve-\f{\kappa}{\mu }\int_{\Omega}\ue\ve^2
+\eps \f{\kappa}{\mu }\io \ue\ln  a\ue
\end{align*}
Because $\eps s(\f{\kappa}{\mu }-s)\ln (\f{\mu }{\kappa}s)$ is negative for any $s>0$, we obtain
\begin{align*}
\calF_\eps'&\le
-\mu\int_{\Omega}\Big(\ue-\f{\kappa}{\mu }\Big)^2-\f{\kappa}{\mu }\int_{\Omega}\frac{|\nabla
\ue|^2}{\ue^2}+\f{\kappa}{2\mu }\int_{\Omega}\frac{|\nabla
\ue|^2}{\ue^2}+\f{\kappa}{2\mu }\int_{\Omega}|\nabla
\ve|^2-\f{\kappa}{\mu }\int_{\Omega}|\nabla
\ve|^2-\f{\kappa}{\mu }\int_{\Omega}\ue\ve^2\nn\\
&=
-\mu\int_{\Omega}\Big(\ue-\f{\kappa}{\mu }\Big)^2-\f{\kappa}{2\mu }\int_{\Omega}\frac{|\nabla
\ue|^2}{\ue^2}-\f{\kappa}{2\mu }\int_{\Omega}|\nabla
\ve|^2-\f{\kappa}{\mu }\int_{\Omega}\ue\ve^2
\end{align*}
on $(0,\infty)$, which implies \eqref{FODI}.
\end{proof}

Building upon \eqref{eq:ueminlimit2.bounded} and the second equation of \eqref{epssys}, we can now acquire decay information about $\ve$:
\begin{lem}\label{lem:vtozero}
Let $\chi >0$, $\kappa>0$, $\mu >0$, let $u_0$ and $v_0$ satisfy \eqref{id} and moreover set $a:=\frac{\mu }{\kappa}$. Then for every $p\in[1,\infty)$ and every $\eta >0$ there is $T>0$ such that for every $t>T$ and every $\eps \in[0,1)$ every global classical solution $(\ue,\ve)$ of \eqref{epssys} satisfies
\begin{equation}\label{eq:vtozerolp}
 \norm[\Lom p]{\ve(\cdot,t)}<\eta .
\end{equation}
\end{lem}
\begin{proof}
 By Lemma \ref{lem:ddtF} we find $c_1>0$ such that for any $\eps\in[0,1)$ we have $\intninf\io \kl{\frac{\kappa}{\mu }-\ue}^2<c_1$.
 Integrating the second equation of \eqref{epssys} shows that
\[
 (0,\infty)\ni t\mapsto \io \ve(\cdot,t)
\]
 is decreasing, and that, moreover,
\[
 \frac{\kappa}{\mu }\intnt\io \ve + \intnt\io\kl{\ue-\frac{\kappa}{\mu }}\ve = \intnt \io \ue\ve \leq \io v_0
\]
 for any $t>0$.
 We conclude that for any $t>0$ and any $\eps \in[0,1)$
\begin{align*}
 \io \ve(\cdot,t) \leq \frac1t\intnt \io \ve&\leq\frac{\mu }{\kappa t} \io v_0 + \frac{\mu }{\kappa t}\intnt\io \kl{\frac{\kappa}{\mu }-\ue}\ve\\
 &\leq \frac{\mu }{\kappa t}\io v_0 + \frac{\mu }{\kappa t}\sqrt{\intnt\io v^2}\sqrt{\intnt\io \kl{\frac{\kappa}{\mu }-\ue}^2}\\
&\leq \frac{\mu }{\kappa t}\io v_0 +\frac{\mu }{\kappa t} \sqrt{\norm[\Lom\infty]{v_0}^2|\Om|t} \sqrt{\intninf\io\kl{\frac{\kappa}{\mu }-\ue}^2}\\
 &\leq \frac{\mu }{\kappa t}\io v_0 +\frac{\mu \sqrt{|\Om|}\norm[\Lom\infty]{v_0}\sqrt{c_1}}{\kappa\sqrt{t}}
\end{align*}
and hence already have that $\io \ve(\cdot,t)$ converges to $0$ as $t\to\infty$, uniformly with respect to $\eps $. In order to obtain \eqref{eq:vtozerolp}, we invoke the additional interpolation
\[
 \norm[\Lom p]{\ve(\cdot,t)}\leq \norm[\Lom\infty]{\ve(\cdot,t)}^{\frac{p-1}p}\norm[\Lom 1]{\ve(\cdot,t)}^{\frac1p}\leq \norm[\Lom\infty]{v_0}^{\frac{p-1}p}\norm[\Lom 1]{\ve(\cdot,t)}^{\frac1p},
\]
valid for any $t>0$.
\end{proof}

A combination of the previous lemmata in this section reveals the large time behaviour of bounded classical solutions:

\begin{lem}\label{lem:convergence.classical}
Let $\kappa>0$, $\mu >0$, $\chi >0$ and let $u_0$, $v_0$ satisfy \eqref{id}. For any solution $(u,v)\in C^{2,1}(\Ombar\times(0,\infty))\cap C^0(\Ombar\times[0,\infty))$ of \eqref{a} that satisfies the boundedness condition \eqref{eq:boundednessconditionforregularity}, we have
\begin{equation}\label{eq:convergencestatement}
 u(\cdot,t)\to \frac{\kappa}{\mu } \quad \text{in } C^0(\Ombar),\quad  v(\cdot,t)\to 0\quad \text{in }C^2(\Ombar).
\end{equation}
as $t\to \infty $.
\end{lem}
\begin{proof} For $j\in\mathbb{N} $ we define
\[
 u_j(x,\tau ):=u(x,j+\tau ), \qquad v_j(x,\tau ):=v(x,j+\tau ),\qquad x\in\Ombar, \tau \in[0,1].
\]
We let $(j_k)_{k\in \mathbb{N}}\subset\mathbb{N} $ be a sequence
satisfying $j_k\to \infty $ as $k\to \infty $. By Lemma
\ref{lem:holder} there are $\alpha \in(0,1)$, $C>0$ such that
\[
 \normm{C^{\alpha,\frac{\alpha }2}(\Ombar\times[0,1])}{u_{j_k}}\leq C, \quad \normm{C^{2+\alpha,1+\frac{\alpha}2}(\Ombar\times[0,1])}{v_{j_k}}\leq C
\]
for all $k\in \mathbb{N} $ and hence there are $u,v\in C^{\alpha
,\frac{\alpha }2}(\Ombar\times[0,1])$ such that $u_{j_{k_l}}\to u$
in $C^0(\Ombar\times[0,1])$ and $v_{j_{k_l}}\to v$ in
$C^2(\Ombar\times[0,1])$ as $l\to \infty $ along a suitable
subsequence. According to \eqref{eq:ueminlimit2.bounded} and
Lemma \ref{lem:vtozero} $u\equiv \frac{\kappa}{\mu }$, $v\equiv
0$. Because every subsequence of $((u_j,v_j))_{j\in\mathbb{N} }$
contains a subsequence converging to $(\f{\kappa}{\mu },0)$, we
conclude that $(u_j,v_j)\to(\f{\kappa}{\mu },0)$ in
$C^0(\Ombar\times[0,1])\times C^2(\Ombar\times[0,1])$ and hence, a
fortiori, \eqref{eq:convergencestatement}.
\end{proof}

\begin{proof}[Proof of Theorem \ref{thm3}]
 The statement of Lemma \ref{lem:convergence.classical} is even slightly stronger than that of Theorem \ref{thm3}.
\end{proof}

\section{Weak solutions}\label{sec:weaksol}
Purpose of this section is the construction of weak solutions to \eqref{a}, in those cases, where Theorem \ref{thm1} is not applicable. To this end let us first state what a weak solution is supposed to be:

\begin{dnt}\label{def:weaksol}
 A weak solution to \eqref{a} for initial data $(u_0,v_0)$ as in \eqref{id} is a pair $(u,v)$ of functions
 \begin{align*}
  u\in L^2_{loc}(\Ombar\times[0,\infty)) \quad \text{ with } \quad \na u \in L^1_{loc}(\Ombar\times [0,\infty)),\\
  v\in L^\infty(\Om\times(0,\infty))\quad \text{ with } \quad \na v \in L^2(\Om\times(0,\infty))
 \end{align*}
 such that, for every $\varphi \in C_0^{\infty}(\Ombar\times[0,\infty))$,
 \begin{align*}
  -\intninf\io u\varphi _t -\io u_0\varphi (\cdot,0) &= -\intninf\io \na u\cdot \na \varphi  + \chi \intninf\io u\na v\cdot\na\varphi +\kappa\intninf\io u\varphi -\mu \intninf\io u^2\varphi \\
  -\intninf\io v\varphi _t -\io v_0\varphi (\cdot,0) &= -\intninf\io \na v\cdot \na \varphi  - \intninf \io uv\varphi
 \end{align*}
 hold true.
\end{dnt}

Some of the estimates neeeded for the compactness arguments in the construction of these weak solutions will spring from the following quasi-energy inequality:

\begin{lem}\label{lem:energyfunctional}
 Let $\mu ,\chi \in(0,\infty)$, $\kappa\in\mathbb{R} $ and let $(u_0,v_0)$ satisfy \eqref{id}.
 There are constants $k_1>0$, $k_2>0$ such that for any $\eps \in(0,1)$ the solution of \eqref{epssys} with $a$ as in \eqref{defa} satisfies
\begin{align}\label{eq:ddtuloguplusnavdv}
 \ddt&\kl{\io \ue\ln  \ue + \frac{\chi }{2} \io \frac{|\na \ve|^2}{\ve}}\nn\\
& + \io \frac{|\na \ue|^2}{\ue} + k_1 \io \frac{|\na \ve|^4}{\ve^3} + k_1 \io \ve|D^2\ln  \ve|^2 + \frac{\mu }2\io \ue^2\ln  \ue+ \eps \io \ue^2\ln  a\ue\ln  \ue \nn\\
 &\leq k_2\io \ve+k_3
\end{align}
on $(0,\infty)$.
\end{lem}
\begin{proof}
According to Lemma \ref{lem:ge.for.positive.eps.or.large.mu}, for any $\eps \in(0,1)$, the solution to \eqref{epssys} is global, and from the second equation of \eqref{epssys} we obtain that

 \begin{align}\label{eq:ddtnavdv}
  \ddt \io \frac{|\na \ve|^2}{\ve} &= 2\io \frac{\na \ve \na \vet}\ve - \io \frac{|\na \ve|^2}{\ve^2}\vet\nn\\
 &=-2\io \frac{\Delta \ve \vet}{\ve} + 2\io \frac{|\na \ve|^2}{\ve^2} \vet -\io \frac{|\na \ve|^2}{\ve^2}\vet\nn\\
&= -2\io \frac{|\Delta \ve|^2}\ve + 2\io \ue\Delta \ve +\io \frac{|\na \ve|^2}{\ve^2}\Delta \ve - \io \frac{|\na \ve|^2}{\ve}\ue\nn\\
 &\leq -2\io \frac{|\Delta \ve|^2}\ve - 2\io \na \ue\cdot \na \ve + \io \frac{|\na \ve|^2}{\ve^2}\Delta \ve\quad\text{on } (0,\infty).
 \end{align}
Here we may rely on Lemma \ref{lem:elementary.estimates} b) to obtain $k_1>0$, $k_2>0$ such that
\[
 \ddt \io \frac{|\na \ve|^2}{\ve}\leq  -2\io \na \ue\cdot \na \ve - \frac{2k_1}{\chi }\io \ve|D^2\ln \ve|^2-\frac{2k_1}{\chi }\io \frac{|\na \ve|^4}{\ve^3} + \frac{2k_2}{\chi } \io \ve \quad \text{on } (0,\infty).
\]
Concerning the entropy term, we compute
\begin{align}\label{eq:ddtulogu}
 \ddt \io \ue\ln  \ue &= \io \uet \ln  \ue + \kappa\io \ue - \mu \io \ue^2 -\eps \io \ue^2\ln  a\ue\nn\\
 &= -\io \frac{|\na \ue|^2}\ue  + \chi \io \na \ue\cdot \na \ve + \kappa\io \ue\ln  \ue - \mu \io \ue^2\ln  \ue-\eps \io \ue^2\ln  a\ue\ln  \ue\nn\\
 &\quad+\kappa\io \ue-\mu \io \ue^2-\eps \io \ue^2\ln  a\ue \quad \text{on } (0,\infty ).
\end{align}
Additionally, $s^2\ln  as>-\frac{1}{2a^2e}$ for all $s\in(0,\infty)$, so that for all $\eps \in(0,1)$
we have $-\eps (s^2\ln  as)<\frac1{2a^2e}$. Since moreover
$\lim_{s\to \infty}(\kappa s-\mu s^2+\kappa s\ln  s-\frac{\mu
}2s^2\ln  s)=-\infty $, we can find $k_3>0$ such that
\[
 \kappa s\ln  s -\frac{\mu }{2}s^2\ln  s +\kappa s-\mu s^2-\eps s^2\ln  as \leq \frac{k_3}{|\Om|}
\]
for any $s\geq 0$ and $\eps \in(0,1)$. Inserting this into the sum of \eqref{eq:ddtulogu} and a multiple of \eqref{eq:ddtnavdv}, we obtain \eqref{eq:ddtuloguplusnavdv}.
\end{proof}

The following lemma serves as collection of the bounds we have prepared:

\begin{lem}\label{lem:bounds}
 Let $\mu >0$, $\chi >0$, $\kappa\in\mathbb{R} $ and suppose that $u_0$, $v_0$ satisfy \eqref{id}.
 Then there is $C>0$ and for any $T>0$ and $q>N$ there is $C(T)>0$ such that for any $\eps \in(0,1)$ the solution $(\ue,\ve)$ of \eqref{epssys} with $a$ as in \eqref{defa} satisfies
\begin{align}
 \int_0^T \io \ue^2\leq C(T)\label{bd:ul2}\\
 \int_0^T \io \frac{|\na \ue|^2}u\leq C(T)\label{bd:nau2u}\\
 \int_0^T \io |\na \ue|^\frac43\leq C(T)\label{bd:nau}\\
 \int_0^T \io \ue^2\ln  a\ue \leq C(T)\label{bd:u2logu}\\
 \int_0^T \io \eps \ue^2(\ln  \ue)\ln  a\ue\leq C(T)\label{bd:epsu2logu2}\\
 \int_0^T \io |\na \ve|^4 \leq C\label{bd:nav4}\\
 \int_0^{\infty } \io |\na \ve|^2\leq C\label{bd:nav}\\
 \norm[ L^\infty(\Om\times (0,\infty ))]{\ve}\leq C\label{bd:v}\\
 \norm[L^2((0,T);(W_0^{1,2}(\Om))^\ast)]{\vet}\leq C(T)\label{bd:vt}\\
 \norm[L^1((0,T);(W_0^{1,2}(\Om))^\ast)]{\uet}\leq C(T)\label{bd:ut}
\end{align}
If, moreover $\kappa>0$, then there is $C>0$ such that for any $\eps \in(0,1)$ the solution $(\ue,\ve)$ of \eqref{epssys} with $a=\f{\mu}{\kappa}$ as in \eqref{defa} satisfies
\begin{equation}
  \int_0^\infty\io \kl{\ue-\f{\kappa}{\mu }}^2 \leq C.\label{bd:uminlimit2}.
\end{equation}
\end{lem}
\begin{proof}
 Bondedness of $\ve$ as in \eqref{bd:v} has been shown in Lemma \ref{lem:normvdecreases}; \eqref{bd:ul2}, \eqref{bd:nau2u}, \eqref{bd:u2logu}, \eqref{bd:epsu2logu2} result from Lemma \ref{lem:energyfunctional} by straightforward integration, as well as \eqref{bd:nav4} if Lemma \ref{lem:normvdecreases} is taken into account. Testing the second equation in \eqref{epssys} by $\ve$, \eqref{bd:nav} is readily obtained. By an application of Hölder's inequality, \eqref{bd:nau} immediately follows from \eqref{bd:ul2} and \eqref{bd:nau2u}. Moreover, \eqref{bd:uminlimit2} is a consequence of \eqref{FODI}.
 For any $\varphi \in C_0^\infty(\Ombar\times[0,T))$ we have
\begin{align*}
 \intnT\io \vet\varphi  &= -\intnT\io \na \varphi \cdot \na \ve - \intnT\io \varphi \ue\ve \\
 &\leq \norm[L^2(\Om\times(0,T))]{\na \varphi }\norm[L^2(\Om\times(0,T))]{\na \ve} + \norm[L^\infty(\Om\times(0,T))]{\ve}\norm[L^2(\Om\times(0,T))]{\ue}\norm[L^2(\Om\times(0,T))]{\varphi }
\end{align*}
and -- by \eqref{bd:ul2}, \eqref{bd:nav}, \eqref{bd:v} -- hence \eqref{bd:vt}. In order to obtain \eqref{bd:ut}, we let $\varphi \in (L^1((0,T);(W_0^{2,q}(\Om))^\ast))^\ast = L^\infty((0,T);W_0^{2,q}(\Om))$ with $\norm[L^\infty((0,T);W_0^{2,q}(\Om))]{\varphi }\leq 1$ and have
\begin{align*}
 \intnT\io \uet \varphi &=\intnT\io \ue \Delta \varphi  + \chi \intnT\io \ue\na \ve\cdot\na\varphi  + \kappa\intnT\io \ue\varphi -\mu \intnT\io \ue^2\varphi +\eps \intnT\io \varphi \ue^2\ln  a\ue\\
 &\leq \norm[L^2(\Om\times(0,T))]{\ue}\norm[L^2(\Om\times(0,T))]{\Delta \varphi }+\chi \norm[L^2(\Om\times(0,T))]{\ue}\norm[L^2(\Om\times(0,T))]{\na v}\norm[L^\infty(\Om\times(0,T))]{\na \varphi } \\
 &+ |\kappa| \norm[L^2(\Om\times(0,T))]{\ue}\norm[L^2(\Om\times(0,T))]{\varphi }+\mu \norm[L^2(\Om\times(0,T))]{\ue}^2\norm[L^\infty(\Om\times(0,T))]{\varphi } \\
&+ \norm[L^\infty(\Om\times(0,T))]{\varphi }\eps \intnT\io \ue^2|\ln  a\ue|,
\end{align*}
 which, due to \eqref{bd:ul2}, \eqref{bd:nav}, \eqref{bd:u2logu}, proves \eqref{bd:ut}.
\end{proof}

By means of compactness arguments, these estimates allow for the construction of weak solutions. This is to be our next undertaking:

\begin{lem}\label{lem:weaksol}
 Let $\mu >0$, $\chi >0$, $\kappa\in\mathbb{R} $ and assume that $u_0$, $v_0$ satisfy \eqref{id}.
 There are a sequence $(\eps _j)_{j\in \mathbb{N}}$, $\eps _j\searrow 0$ and functions
\begin{align*}
 u&\in L^2_{loc}(\Ombar\times[0,\infty)) \quad \text{ with } \quad \na u\in L^{\frac43}_{loc}(\Ombar\times [0,\infty)),\\
 v&\in L^\infty(\Om\times(0,\infty)) \quad \text{ with }\quad \na v \in L^2(\Om\times(0,\infty))
\end{align*}
such that the solutions $(\ue,\ve)$ of \eqref{epssys} with $a$ as in \eqref{defa} satisfy
\begin{align}
 \ue&\to u & &\text{in } L^{\frac 43}_{loc}([0,\infty);L^{\frac43}(\Om)) \quad \text{ and a.e. in } \Om\times(0,\infty)\label{conv:u}\\
 \na\ue&\weakto \na u&&\text{ in } L^{\frac 43}_{loc}([0,\infty);L^{\frac43}(\Om))\label{conv:nau}\\
 \ue^2&\to u^2&&\text{ in } L^1_{loc}(\Ombar\times[0,\infty))\label{conv:u2}\\
 \eps \ue^2\ln (a\ue)&\to 0&&\text{ in } L^1_{loc}(\Ombar\times[0,\infty))\label{conv:epsu2logu}\\
 \ve &\to v & & \text{a.e. in } \Om\times (0,\infty)\label{conv:v}\\
 \ve &\weakstarto v && \text{in } L^\infty((0,\infty),\Lom p)\quad \text{for any } p\in[1,\infty]\label{conv:vweakstar}\\
 \na \ve&\weakto \na v && \text{in } L^4_{loc}([0,\infty);\Lom4)\label{conv:nav4}\\
 \na\ve&\weakto \na v && \text{in } L^2((0,\infty);\Lom2)\label{conv:nav}
\end{align}
as $\eps =\eps _j\searrow 0$ and such that $(u,v)$ is a weak solution to \eqref{a}.\\
If additionally $\kappa>0$ and $a=\f{\mu }{\kappa}$ as in \eqref{defa}, then $\eps _j$ can be chosen such that additionally
\begin{equation}
 \ue-\frac{\kappa}{\mu }\weakto u-\f{\kappa}{\mu }\qquad\text{ in  } L^2(\Om\times(0,\infty))\label{conv:uminlimit}
\end{equation}
as $\eps =\eps _j\searrow 0$, and
\begin{equation}\label{uminlimitinL2}
 \kl{u-\frac{\kappa}{\mu }}\in L^2(\Om\times(0,\infty))
\end{equation}
\end{lem}
\begin{proof}
 \cite[Cor. 8.4]{simon} transforms \eqref{bd:ul2}, \eqref{bd:nau} and \eqref{bd:ut} into \eqref{conv:u} along a suitable sequence $(\eps_j)_j\searrow 0$; the bound in \eqref{bd:nau} enables us to find a further subsequence such that \eqref{conv:nau} holds. Similarly, \eqref{bd:nav} facilitates the extraction of a subsequence satisfying \eqref{conv:nav}, and an analogous application of \cite[Cor. 8.4]{simon} as before from \eqref{bd:nav}, \eqref{bd:v} and \eqref{bd:vt} provides a (non-relabeled) subsequence such that $v_{\eps _j}\to v$ in $L^2(\Om\times(0,\infty))$ and, along another subsequence thereof establishes \eqref{conv:v}. Also \eqref{conv:vweakstar} is immediately obtained from \eqref{bd:v}, as is \eqref{conv:nav4} from \eqref{bd:nav4}; \eqref{conv:uminlimit} results from \eqref{bd:uminlimit2}. For the $L^1$-convergence statements in \eqref{conv:u2} and \eqref{conv:epsu2logu}, mere boundedness, like obtainable from \eqref{bd:nau2u} and \eqref{bd:u2logu}, even if combined with the a.e. convergence provided by \eqref{conv:u}, is insufficient for the existence of a convergent  subsequence; we must, in addition,
check for equi-integrability on $\Om\times(0,T)$ for any
finite $T>0$. To this purpose we note that with $C(T)$ from \eqref{bd:epsu2logu2}
\begin{align*}
 \inf_{b\geq0} \sup_{\eps \in(0,1)} \intnT\int_{\set{\eps \ue^2\ln a \ue>b}} |\eps \ue^2\ln  a \ue| &\leq \inf_{b>a }\sup_{\eps \in(0,1)} \intnT\int_{\set{\eps \ue^2\ln a \ue>b}} \eps \ue^2\ln  a \ue\\
  &\leq \inf_{b>a }\sup_{\eps \in(0,1)} \intnT\int_{\set{a \ue^3>b}} \eps \ue^2\ln  a \ue\\
  &\leq \inf_{b>a }\sup_{\eps \in(0,1)} \intnT\int_{\set{\ln  \ue>\frac13 \ln  \frac ba}} \eps \ue^2(\ln  a\ue)\ln  \ue\cdot \frac 3{\ln  \frac{b}{a}}\\
 &\leq \inf_{b>a } \frac{3C(T)}{\ln  \frac{b}{a}} =0
\end{align*}
and, due to \eqref{bd:u2logu},
\begin{align*}
  \inf_{b\geq 0} \sup_{\eps \in(0,1)} \intnT\int_{\set{\ue^2>b}}\ue^2 \leq \inf_{b>1}\sup_{\eps \in(0,1)}\intnT\int_{\set{\ue^2>b}} \ue^2\ln  \ue \frac{1}{\ln  b} \leq \inf_{b>1} \frac{C(T)}{\ln  b} = 0.
\end{align*}
Accordingly, $\set{ \eps \ue^2\ln \ue; \eps \in(0,1)}$ and $\set{\ue^2; \eps \in(0,1)}$ are uniformly integrable, hence 
by \eqref{conv:u} and the Vitali convergence theorem we can extract subsequences such that \eqref{conv:epsu2logu} and \eqref{conv:u2} hold; \eqref{conv:u2} also proves that $u\in L^2_{loc}(\Ombar\times[0,\infty ))$. Passing to the limit in each of the integrals making up a weak formulation of \eqref{epssys} with $\eps >0$, which is possible due to \eqref{conv:u}, \eqref{conv:nau}, \eqref{conv:nav4}, \eqref{conv:u2}, \eqref{conv:epsu2logu} and \eqref{conv:vweakstar}, shows that $(u,v)$ is a weak solution to \eqref{epssys} with $\eps =0$.
\end{proof}

\begin{proof}[Proof of Theorem \ref{thm:weaksol}]
 The assertion of Theorem \ref{thm:weaksol} is part of Lemma \ref{lem:weaksol}.
\end{proof}

We will finally prove that one can expect at least some stabilization of weak solutions also. Here, the preparation in Lemma \ref{lem:vtozero} obtained from the energy inequality for $\calF$ will be crucial. 

\begin{lem}\label{lem:weaklimit} Let $\mu >0$, $\chi >0$, $\kappa>0$ and assume that $u_0$, $v_0$ satisfy \eqref{id}.
 The weak solution $(u,v)$ to \eqref{a} obtained in Lemma \ref{lem:weaksol} satisfies
 \begin{equation}\label{eq:weaksolvtozero}
  \norm[\Lom p]{v(\cdot,t)}\to 0
 \end{equation}
 for any $p\in[1,\infty)$ and
\begin{equation}\label{eq:weaksoluconv}
 \int_t^{t+1} \norm[\Lom 2]{u-\frac{\kappa}{\mu }}\to0
\end{equation}
 as $t\to \infty$.
\end{lem}
\begin{proof}
 Using characteristic functions of sets $\Om\times(t,t+1)$ for sufficiently large $t$ as test functions in the weak-$*$-convergence statement \eqref{conv:vweakstar}, from Lemma \ref{lem:vtozero} we obtain that for every $\eta>0$ there is $T>0$ such that $\norm[L^\infty((T,\infty);\Lom p)]{v}<\eta$,
whereas \eqref{eq:weaksoluconv} is implied by \eqref{uminlimitinL2}.
\end{proof}

\begin{remark}
 If $N\leq 3$, the uniform bound on $\int_t^{t+1}\io |\na v|^4$ contained in Lemma \ref{lem:energyfunctional} proves to be sufficient for \eqref{eq:weaksolvtozero} even to hold for $p=\infty$, which can be used as starting point for derivation of eventual smoothness of solutions via a quasi-energy-inequality for $\io \frac{u^p}{(\eta -v)^\theta}$ with suitable numbers $\theta$ and $\eta $. This result is already contained in \cite{lankeit_fluid}.
\end{remark}

\begin{proof}[Proof of Theorem \ref{thm:weaksol-limit}]
 Lemma \ref{lem:weaklimit} is identical with Theorem \ref{thm:weaksol-limit}.
\end{proof}

\section*{Acknowledgment}

 J.~Lankeit acknowledges support of the {\em Deutsche Forschungsgemeinschaft} within the project {\em Analysis of chemotactic cross-diffusion in complex
 frameworks}. Y.~Wang was supported by the NNSF of China (no. 11501457).

{\small

\begin{thebibliography}{10}

\bibitem{BBTW}
N.~Bellomo, A.~Bellouquid, Y.~Tao, and M.~Winkler.
\newblock Toward a mathematical theory of {K}eller-{S}egel models of pattern
  formation in biological tissues.
\newblock {\em Math. Models Methods Appl. Sci.}, 25(9):1663--1763, 2015.

\bibitem{cao_lankeit}
X.~Cao and J.~Lankeit.
\newblock Global classical small-data solutions for a three-dimensional
  chemotaxis {N}avier-{S}tokes system involving matrix-valued sensitivities.
\newblock {\em Calc. Var. Partial Differential Equations}, 2016.
\newblock to appear.

\bibitem{he_zheng}
X.~He and S.~Zheng.
\newblock Convergence rate estimates of solutions in a higher dimensional
  chemotaxis system with logistic source.
\newblock {\em J. Math. Anal. Appl.}, 436(2):970--982, 2016.

\bibitem{horstmann_I}
D.~Horstmann.
\newblock From 1970 until present: the {K}eller-{S}egel model in chemotaxis and
  its consequences. {I}.
\newblock {\em Jahresber. Deutsch. Math.-Verein.}, 105(3):103--165, 2003.

\bibitem{KS}
E.~F. {K}eller and L.~A. {S}egel.
\newblock Initiation of slime mold aggregation viewed as an instability.
\newblock {\em J. Theoret. Biol.}, 26(3):399 -- 415, 1970.

\bibitem{LSU}
O.~A. Lady{\v{z}}enskaja, V.~A. Solonnikov, and N.~N. Ural'ceva.
\newblock {\em Linear and quasilinear equations of parabolic type}.
\newblock Translated from the Russian by S. Smith. Translations of Mathematical
  Monographs, Vol. 23. American Mathematical Society, Providence, R.I., 1968.

\bibitem{lankeit_ev_smooth}
J.~Lankeit.
\newblock Eventual smoothness and asymptotics in a three-dimensional chemotaxis
  system with logistic source.
\newblock {\em J. Differential Equations}, 258(4):1158--1191, 2015.

\bibitem{lankeit_fluid}
J.~Lankeit.
\newblock Long-term behaviour in a chemotaxis-fluid system with logistic
  source.
\newblock {\em Math. Models Methods Appl. Sci.}, 2016.
\newblock to appear.

\bibitem{xieli}
X.~Li.
\newblock Global existence and uniform boundedness of smooth solutions to a
  parabolic-parabolic chemotaxis system with nonlinear diffusion.
\newblock {\em Bound. Value Probl.}, pages 2015:107, 17, 2015.

\bibitem{lin_mu}
K.~Lin and C.~Mu.
\newblock Global dynamics in a fully parabolic chemotaxis system with logistic
  source.
\newblock {\em Discrete and Continuous Dynamical Systems}, 36(9):5025--5046,
  2016.

\bibitem{mizoguchi_winkler_13}
N.~Mizoguchi and M.~Winkler.
\newblock Blow-up in the two-dimensional parabolic {K}eller-{S}egel system.
\newblock 2013.
\newblock Preprint.

\bibitem{osaki_yagi_02}
K.~Osaki and A.~Yagi.
\newblock Global existence for a chemotaxis-growth system in $\mathbb{R}^{2}$.
\newblock {\em Advances in Mathematical Sciences and Applications},
  12(2):587--606, 2002.

\bibitem{porzio_vespri}
M.~M. Porzio and V.~Vespri.
\newblock H\"older estimates for local solutions of some doubly nonlinear
  degenerate parabolic equations.
\newblock {\em J. Differential Equations}, 103(1):146--178, 1993.

\bibitem{simon}
J.~Simon.
\newblock Compact sets in the space {$L^p(0,T;B)$}.
\newblock {\em Ann. Mat. Pura Appl. (4)}, 146:65--96, 1987.

\bibitem{tao_consumption_bdness}
Y.~Tao.
\newblock Boundedness in a chemotaxis model with oxygen consumption by
  bacteria.
\newblock {\em J. Math. Anal. Appl.}, 381(2):521--529, 2011.

\bibitem{taowin_ev_consumption}
Y.~Tao and M.~Winkler.
\newblock Eventual smoothness and stabilization of large-data solutions in a
  three-dimensional chemotaxis system with consumption of chemoattractant.
\newblock {\em J. Differential Equations}, 252(3):2520--2543, 2012.

\bibitem{taowin_persistence}
Y.~Tao and M.~Winkler.
\newblock Persistence of mass in a chemotaxis system with logistic source.
\newblock {\em J. Differential Equations}, 259(11):6142--6161, 2015.

\bibitem{wang_khan_khan}
L.~Wang, S.~U.-D. Khan, and S.~U.-D. Khan.
\newblock Boundedness in a chemotaxis system with consumption of
  chemoattractant and logistic source.
\newblock {\em Electron. J. Differential Equations}, pages No. 209, 9, 2013.

\bibitem{win_ksns_logsource}
M.~Winkler.
\newblock A three-dimensional Keller-Segel-Navier-Stokes system
with logistic source: Global weak solutions and asymptotic
stabilization.
  \newblock {\em Preprint}.

\bibitem{win_aggregationvs}
M.~Winkler.
\newblock Aggregation vs. global diffusive behavior in the higher-dimensional
  {K}eller-{S}egel model.
\newblock {\em J. Differential Equations}, 248(12):2889--2905, 2010.

\bibitem{winkler_10_boundedness}
M.~Winkler.
\newblock Boundedness in the higher-dimensional parabolic-parabolic chemotaxis
  system with logistic source.
\newblock {\em Comm. Partial Differential Equations}, 35(8):1516--1537, 2010.

\bibitem{win_ctfluid}
M.~Winkler.
\newblock Global large-data solutions in a chemotaxis-({N}avier-){S}tokes
  system modeling cellular swimming in fluid drops.
\newblock {\em Comm. Partial Differential Equations}, 37(2):319--351, 2012.

\bibitem{win_blowuphigherdim}
M.~Winkler.
\newblock Finite-time blow-up in the higher-dimensional parabolic-parabolic
  {K}eller-{S}egel system.
\newblock {\em J. Math. Pures Appl. (9)}, 100(5):748--767, 2013.

\bibitem{win_stability}
M.~Winkler.
\newblock Global asymptotic stability of constant equilibria in a fully
  parabolic chemotaxis system with strong logistic dampening.
\newblock {\em J. Differential Equations}, 257(4):1056--1077, 2014.

\bibitem{win_arma}
M.~Winkler.
\newblock Stabilization in a two-dimensional chemotaxis-{N}avier-{S}tokes
  system.
\newblock {\em Arch. Ration. Mech. Anal.}, 211(2):455--487, 2014.

\bibitem{zhang_li}
Q.~Zhang and Y.~Li.
\newblock Stabilization and convergence rate in a chemotaxis system with
  consumption of chemoattractant.
\newblock {\em J. Math. Phys.}, 56(8):081506, 10, 2015.

\end{thebibliography}

}
\end{document}